\documentclass[12pt]{article}

\usepackage{amsmath}\usepackage{amssymb}\usepackage{amsfonts}\usepackage{amsthm}
\usepackage{color,url}
\usepackage{setspace}
\usepackage{graphicx}
\usepackage[english]{babel}\usepackage[latin1]{inputenc}\usepackage{times}\usepackage[T1]{fontenc}
\usepackage{epstopdf}
\usepackage{epsfig}
\usepackage{float}

\newtheorem{thm}{\bf{Theorem}}[section]
\newtheorem{lem}[thm]{\bf{Lemma}}

\newtheorem{df}[thm]{\bf{Definition}}
\newtheorem{cor}[thm]{\bf{Corollary}}
\newtheorem{rem}[thm]{\bf{Remark}}

\newtheorem{prop}[thm]{\bf{Proposition}}
\newtheorem{fact}[thm]{\bf{Fact}}
\newtheorem{ex}[thm]{\bf{Example}}

\numberwithin{equation}{section}

\newcommand{\dom}{\operatorname{dom}}
\newcommand{\ran}{\operatorname{ran}}
\newcommand{\intt}{\operatorname{int}}

\newcommand{\R}{\operatorname{\mathbb{R}}}

\newcommand{\OR}{\operatorname{\overline{\R}}}

\newcommand{\Prox}{\operatorname{Prox}}

\title{Conditions for the existence, identification and calculus rules of the threshold of prox-boundedness}

\author{C. Planiden\thanks{School of Mathematics and Applied Statistics, University of Wollongong, Wollongong, NSW, 2522, Australia. chayne@uow.edu.au}}
\date{\today}
\begin{document}

\maketitle\author
\setcounter{page}{1}\pagenumbering{arabic}

\begin{abstract}
This work advances knowledge of the threshold of prox-boundedness of a function; an important concern in the use of proximal point optimization algorithms and in determining the existence of the Moreau envelope of the function. In finite dimensions, we study general prox-bounded functions and then focus on some useful classes such as piecewise functions and Lipschitz continuous functions. The thresholds are explicitly determined when possible and bounds are established otherwise. Some calculus rules are constructed; we consider functions with known thresholds and find the thresholds of their sum and composition.
\end{abstract}

\noindent \textbf{2000 Mathematics Subject Classification:} \medskip\\
Primary 49J53; Secondary 26A06, 90C30\medskip

\section{Introduction}\label{sec:intro}

The Moreau envelope function, also known as Moreau--Yosida regularization, is a particular infimal convolution that first came about in the 1960s \cite{proximite}. Given a function $f$ on a finite-dimensional space, the Moreau envelope of $f$ employs a nonnegative parameter $r$ and is denoted $e_rf$:
\begin{equation}\label{eq0}e_rf(x)=\inf\limits_{y\in\R^n}\left\{f(y)+\frac{r}{2}\|y-x\|^2\right\}.\end{equation}It is a well-established, regularizing function that has many desirable properties when $f$ has reasonable structure \cite{rockwets,funcanal}:
\begin{itemize}
\item if $f$ is convex nonsmooth, $e_rf$ is convex smooth and there is an explicit formula for the gradient $\nabla e_rf$;
\item the functions $f$ and $e_rf$ have the same minimum and minimizers when $f$ is convex;
\item as $r\to\infty$, $e_rf\to f$.
\end{itemize}The set of all solution points to \eqref{eq0} is known as the \emph{proximal mapping} of $f$, denoted by $\Prox_f$. The proximal mapping is a key component of many optimization algorithms, such as the proximal point method and its variants \cite{proxpoint,MR1735448,MR1168183,MR628084,hare2018computing,hare2019derivative,monops}. Because of the above and other nice features, the Moreau envelope and proximal mapping have been thoroughly researched and applied to many situations in the convex \cite{MR3014983,ontheconv,MR3143754,martreg,MR3513870,MR3931109} and nonconvex \cite{compprox,ppm,MR3392363,fastmoreau,vusmoothness,MR3783644,ppa} settings.\par For a particular function $f$, its Moreau envelope may or may not exist, or may exist only for certain $x$ and/or certain $r$. If there does exist a point $x$ such that $e_rf(x)\in\R$ for some $r\geq0,$ we say that $f$ is \emph{prox-bounded}. If a function is not prox-bounded, then its Moreau envelope does not exist anywhere, for any choice of $r$. It is the parameter $r$ that is of primary interest in this work. There are many theoretically proved-convergent proximal algorithms (see \cite{annergren2012admm,MR3935085,briceno2011proximal,calamai1987projected,MR3921395,douglas1956numerical} and the references therein), but in practice, it has been observed that the initial choice of $r$ and the manner of adjusting it as the algorithm runs are of critical importance, in order to obtain reliable performance \cite{boyd2011distributed,fougner2018parameter,MR2275356,he2000alternating,MR3853209,MR1928047}. We explore the \emph{threshold of prox-boundedness} of $f$: the infimum of the set of $r\geq0$ such that $e_rf$ exists at at least one point.\par In \cite{hare2014thresholds}, the class of piecewise linear-quadratic (PLQ) functions on $\R^n$ was examined in the context of prox-boundedness. The main result of \cite{hare2014thresholds} is a theorem that explicitly identifies the threshold and the domain of the Moreau envelope of any finite-dimensional PLQ function. In that setting, the threshold is $\max r_i,$ where $r_i$ is the threshold of $f_i$ for each $i$. One of the aims of the present work is to generalize that result in two aspects. We consider the cases where
\begin{enumerate}
\item[(1)]the functions $f_i$ are not necessarily linear nor quadratic and
\item[(2)]the domains $\dom f_i$ are not necessarily polyhedral. 
\end{enumerate}The main question on which we focus is this: what are the minimal conditions needed on $f_i$ in order to be sure that the threshold of the piecewise function is $\max r_i?$ We establish bounds and illustrate several counterexamples for functions with conditions that one might suspect sufficient to guarantee prox-boundedness, but are not. Under certain conditions, the threshold can be determined exactly. \par The second focus of this work is the establishing of calculus rules for thresholds of prox-bounded functions. We explore classes of functions with known thresholds and study the conditions needed to determine the threshold of their sum and composition, or to produce an upper bound when an exact result cannot be found. By making use of the Fenchel conjugate representation of the Moreau envelope and some other previously-established properties and characterizations of prox-bounded functions, we determine sufficient conditions for the existence of a sum rule and a composition rule for thresholds of prox-boundedness. \par The remainder of this work is organized as follows. Section \ref{sec:prelim} presents the notation used throughout and several definitions and known facts regarding prox-bounded functions and their thresholds. In Section \ref{sec:main}, we examine the family of piecewise functions and determine the minimal conditions needed for finding the threshold. An example of what can go wrong when these conditions are not met is provided. Section \ref{sec:calculus} is dedicated to forming calculus rules for the threshold of the sum and the composition of prox-bounded functions. Section \ref{sec:conc} offers concluding remarks and suggests interesting areas of further research in this vein.

\section{Preliminaries}\label{sec:prelim}

\subsection{Notation}

Throughout this paper, we work in finite-dimensional space $\R^n,$ endowed with inner product $\langle\cdot,\cdot\rangle$ and induced norm $\|\cdot\|.$ The set $\R^n\cup\{+\infty\}$ is denoted by $\OR$. We generally conform to the notation used in \cite{rockwets}, including the terms proper and lower semicontinuous (lsc) defined therein. We denote the domain of a function $f$ by $\dom f$ and the gradient of $f$ by $\nabla f.$ The distance from a point $x\in\R^n$ to a set $C$ is denoted $d_C(x),$ and the projection of $x$ onto $C$ is denoted $P_Cx$. 
\begin{df}
Given $K>0$, the function $f:\R^n\to\R$ is \emph{locally $K$-Lipschitz continuous} about $z\in\dom f$ with radius $\sigma > 0$ if
$$\|f(y)-f(x)\|\leq K\|y-x\|\mbox{ for all }x,y\in B_\sigma(z),$$
and $f$ is \emph{globally $K$-Lipschitz continuous} if $\sigma$ can be taken to be $\infty$.
\end{df}
\begin{df}
The \emph{indicator function} of a set $S$ is denoted $\iota_S$ and defined by
$$\iota_s(x)=\begin{cases}
0,&x\in S,\\\infty,&x\not\in S.
\end{cases}$$
\end{df}
\begin{df}\label{df:piecewise}
For proper functions $f_i:\R^n\rightarrow\OR,$ $i\in\{1,2,\ldots,m\},$ the \emph{piecewise function} $f:\R^n\rightarrow\OR$ is defined by
$$f(x)=\begin{cases}
f_1(x),&x\in S_1,\\
&\vdots\\
f_m(x),&x\in S_m,
\end{cases}$$
where $\bigcup_iS_i=\R^n$ and $S_i\cap\intt S_j=\varnothing$ for $i\neq j.$
\end{df}
\noindent Notice that a piecewise function is not necessarily continuous, as we do not require $f_i(x)=f_j(x)$ on $S_i\cap S_j.$
\begin{df}
Given a finite number of functions $f_1,f_2,\ldots,f_m:\R^n\rightarrow\OR,$ the \emph{finite-max function} $f$ is defined by
$$f(x)=\max\{f_i(x)\}.$$The \emph{active set} of indices for $f$ at a point $x$ is the set
$$A_x=\{i:f_i(x)=f(x)\}.$$
\end{df}
\noindent Note that a finite-max function is a piecewise function with $S_i=\{x:i\in A_x\}.$
\begin{df}
A function $f:\R^n\rightarrow\OR$ is \emph{prox-bounded} if there exists $r\geq0$ such that $e_rf(x)>-\infty$ for some $x\in\R^n.$ The infimum of all such $r$ is called the \emph{threshold of prox-boundedness} of $f.$
\end{df}
Our interest in this work is to identify thresholds of prox-boundedness of functions and the conditions for their existence. To that end, we list the following results from previous works.
\begin{fact}\emph{\cite[Lemma 2.4]{hare2014thresholds}}\label{fact1}
Let $f:\R^n\rightarrow\OR$ be proper and lsc. Then $f$ is bounded below if and only if its threshold $\bar{r}=0$ and $\dom e_{\bar{r}}f=\R^n.$
\end{fact}
\begin{fact}\emph{\cite[Theorem 1.25]{rockwets}}\label{fact2}
Let $f:\R^n\to\R$ be proper, lsc and prox-bounded with threshold $\bar{r}$. Then for all $r>\bar{r}$, $\dom e_rf=\R^n$.
\end{fact}
\begin{fact}\emph{\cite[Example 3.28]{rockwets}}\label{fact3}
Let $f:\R^n\to\R$ be such that $$\liminf\limits_{\|x\|\to\infty}\frac{f(x)}{\|x\|}>-\infty.$$Then $f$ is prox-bounded with threshold $\bar{r}=0$.
\end{fact}
\begin{fact}\emph{\cite[Theorem 2.26]{rockwets}}\label{fact4}
Let $f:\R^n\to\R$ be proper, lsc and convex. Then $f$ is prox-bounded with threshold $\bar{r}=0$.
\end{fact}
\noindent Note that by definition of prox-regular, if $r<\bar{r}$, then $e_rf(x)=-\infty$ for all $x\in\R^n$. By Fact \ref{fact2}, if $r>\bar{r}$, then $e_rf(x)>-\infty$ for all $x\in\R^n$. At the threshold itself, however, there is no such universal behaviour of the Moreau envelope. Depending on the nature of $f$, $\dom e_{\bar{r}}f$ can be empty, full-domain or a proper subset of $\R^n$, even for very simple functions (see \cite[Examples 2.5--2.7]{hare2014thresholds}). This is partly why the choice of initial prox-parameter $r$ and the manner in which it changes are so crucial in many minimization algorithms; vastly different proximal behaviour is possible with distinct values of $r$. In the next section, we explore existence conditions for the threshold of prox-boundedness of piecewise functions.

\section{Existence of thresholds of prox-boundedness}\label{sec:main}

The conditions for prox-boundedness in the case of PLQ functions was thoroughly examined in \cite{hare2014thresholds}. Now we move beyond that class of functions. We concentrate primarily on piecewise functions as defined in Definition \ref{df:piecewise}, as that is a natural extension to what has been accomplished already. We begin by establishing the fact that Lipschitz continuous functions are prox-bounded.
\begin{prop}\label{prop0}
Let $f:\R^n\to\R$ be proper and globally $K$-Lipschitz. Then $f$ is prox-bounded with threshold $\bar{r}=0$.
\end{prop}
\begin{proof}
Suppose that $f$ is not prox-bounded, i.e.\ $e_rf(x)=-\infty$ for all $r\geq0$, for all $x\in\R^n$. Let $r>0$ be fixed and arbitrary. Then for any $\bar{x}\in\R^n$, there exists a sequence $\{x_\alpha\}_{\alpha=1}^\infty\subseteq\R^n$ such that 
\begin{equation}\label{eq1}\lim\limits_{\alpha\to\infty}\left\{f(x_\alpha)+\frac{r}{2}\|\bar{x}-x_\alpha\|^2\right\}=-\infty.\end{equation}Then $\lim_{\alpha\to\infty}f(x_\alpha)=-\infty$, since the other term $\frac{r}{2}\|\bar{x}-x_\alpha\|^2$ is always nonnegative. Hence,$$\lim_{\alpha\to\infty}|f(\bar{x})-f(x_\alpha)|=\infty\quad(f \mbox{ is proper, so }f(\bar{x})\neq-\infty),$$and since $f$ is $K$-Lipschitz, we have$$\lim\limits_{\alpha\to\infty}K\|\bar{x}-x_\alpha\|=\infty.$$Thus, $\lim_{\alpha\to\infty}\frac{r}{2}\|\bar{x}-x_\alpha\|^2=\infty$. This together with \eqref{eq1} says that as $\alpha\to\infty$, $f(x_\alpha)\to-\infty$ faster than $\frac{r}{2}\|\bar{x}-x_\alpha\|^2\to\infty$, i.e., $\frac{r}{2}\|\bar{x}-x_\alpha\|^2=o(f(x_\alpha))$. Since $\frac{r}{2}\|\bar{x}-x_\alpha\|^2$ is not a Lipschitz continuous function and $f(x_\alpha)$ grows even faster, we have that $f(x_\alpha)$ is not Lipschitz either, a contradiction. Therefore, there must exist at least one $\bar{x}\in\R^n$ such that $e_rf(\bar{x})>-\infty$ and we have that $f$ is prox-bounded. Since this is true for any arbitrary $r>0$, it is true for all $r>0$. The threshold of prox-boundedness is the infimum of all such $r$, so $\bar{r}=0$.
\end{proof}
Now we focus on the family of piecewise functions and say what we can about their thresholds. Henceforth, any mention of a piecewise function refers to a function defined as in Definition \ref{df:piecewise}. 
\begin{prop}\label{prop1}
Let $f:\R^n\to\OR$ be a proper, lsc, piecewise function. Then $f$ is prox-bounded if and only if $f_i+\iota_{S_i}$ is prox-bounded for each $i.$
\end{prop}
\begin{proof}
\begin{itemize}
\item[$(\Rightarrow)$] Let $f$ be prox-bounded with threshold $\bar{r}.$ Suppose that there exists $j$ such that $f_j+\iota_{S_j}$ is not prox-bounded. Fix $r>\bar{r},$ so that $e_rf(x)\in\R$ for all $x$ (Fact \ref{fact1}). By definition of $e_rf,$ for $\bar{x}$ fixed we have
\footnotesize\begin{align}
e_rf(\bar{x})&=\inf\limits_y\left\{f(y)+\frac{r}{2}\|y-\bar{x}\|^2\right\}\nonumber\\
&=\min\left[\inf\limits_y\left\{f_1(y)+\iota_{S_1}(y)+\frac{r}{2}\|y-\bar{x}\|^2\right\},\ldots,\inf\limits_y\left\{f_m(y)+\iota_{S_m}(y)+\frac{r}{2}\|y-\bar{x}\|^2\right\}\right].\label{pb1}
\end{align}\normalsize
Since $f_j+\iota_{s_j}$ is not prox-bounded, we have 
$$\inf\limits_y\left\{f_j(y)+\iota_{S_j}(y)+\frac{r}{2}\|y-\bar{x}\|^2\right\}=-\infty,$$
hence $e_rf(\bar{x})=-\infty.$ This is a contradiction to the fact that $e_r(\bar{x})\in\R.$ Therefore, $f_i+\iota_{S_i}$ is prox-bounded for all $i\in\{1,2,\ldots,m\}.$
\item[$(\Leftarrow)$] Let $f_i+\iota_{S_i}$ be prox-bounded with threshold $r_i$ for each $i\in\{1,2,\ldots,m\}.$ Let $\bar{r}=\max r_i,$ and choose any $r>\bar{r}.$ Then $r>r_i$ for all $i,$ so that $e_rf_i(x)\in\R$ for all $i$ and for all $x.$ Then each infimum in \eqref{pb1} is a real number, hence the minimum of \eqref{pb1} exists. Therefore, $e_rf(\bar{x})\in\R,$ and we have that $f$ is prox-bounded.\qedhere
\end{itemize}
\end{proof}
\begin{thm}\label{thm1}
Let $f:\R^n\to\OR$ be a proper, lsc, piecewise function and let each $f_i+\iota_{S_i}$ be prox-bounded with threshold $r_i.$ Then the threshold of $f$ is $\bar{r}=\max r_i.$
\end{thm}
\begin{proof}
Choose any $r>\max r_i.$ By Proposition \ref{prop1}, there exists $\bar{x}$ such that $e_rf(\bar{x})\in\R.$ Since $r$ is arbitrary, we have that $f$ is prox-bounded for all $r>\max r_i.$ Hence, $\bar{r}\leq\max r_i.$ Now choose any $r<\max r_i.$ Then there exists $j\in\{1,2,\ldots,m\}$ such that $r<r_j.$ By definition of prox-boundedness, $e_r\left(f_j+\iota_{S_j}\right)(x)=-\infty$ for all $x.$ The Moreau envelope of $f$ is the expression of \eqref{pb1}, whose minimand contains at least one instance of $-\infty$ due to $f_j+\iota_{S_j}.$ Hence, $e_rf(x)=-\infty$ for all $x,$ and we have that $\bar{r}\geq\max r_i.$ Therefore, $\bar{r}=\max r_i.$
\end{proof}
We have our first results for piecewise functions. However, the result of Theorem \ref{thm1} is weakened if the term $\iota_{S_i}$ is removed from the statement and we require $f_i$ itself to be prox-bounded, as Theorem \ref{cor1} shows.
\begin{thm}\label{cor1}
For $i\in\{1,2,\ldots,m\},$ let $f_i:\R^n\rightarrow\OR$ be proper, lsc and prox-bounded with threshold $r_i.$ With these $f_i,$ define a piecewise function $f$ as per Definition \ref{df:piecewise}. Then $f$ is prox-bounded with threshold $\bar{r}\leq\max r_i.$
\end{thm}
\begin{proof}
Since $f_i$ is prox-bounded with threshold $r_i,$ there exists $\bar{x}\in\R^n$ such that $e_{r_i}f_i(\bar{x})>-\infty.$ Since 
$$e_{r_i}f_i(\bar{x})=\inf\limits_{y\in\R^n}\left\{f_i(y)+\frac{r_i}{2}\|y-\bar{x}\|^2\right\}\leq\inf\limits_{y\in S_i}\left\{f_i(y)+\frac{r_i}{2}\|y-\bar{x}\|^2\right\}=e_{r_i}(f_i+\iota_{S_i})(\bar{x}),$$
we have that $e_{r_i}(f_i+\iota_{S_i})(\bar{x})>-\infty.$ Hence, $f_i+\iota_{S_i}$ is prox-bounded with threshold $\tilde{r}_i\leq r_i.$ This is true for all $i\in\{1,2,\ldots,m\},$ so Theorem \ref{thm1} applies and we have 
$\bar{r}=\max\tilde{r}_i\leq\max r_i.$
\end{proof}
\noindent The best we can do is an upper bound in this case. One might hope to establish a lower bound for $\bar{r}$ as well, such as $\min r_i.$ However, this cannot be done in the general setting of Theorem \ref{cor1}. Example \ref{ex1} illustrates why not.
\begin{ex}\label{ex1}
Let $f_1,f_2:\R\rightarrow\R,$
$$f_1(x)=\begin{cases}x^2,&x<0,\\-x^2,&x\geq0,\end{cases}\qquad f_2(x)=\begin{cases}-x^2,&x<0,\\x^2,&x\geq0.\end{cases}$$
Define
$$f(x)=\begin{cases}f_1(x),&x<0,\\f_2(x),&x\geq0.\end{cases}$$
Then $r_1=r_2=2$ and $\bar{r}=0.$
\end{ex}
\begin{proof}
Considering $f_1$ first, we define
$$\varphi_r(y)=f_1(y)+\frac{r}{2}|y-x|^2=\begin{cases}\left(1+\frac{r}{2}\right)y^2-rxy+\frac{r}{2}x^2,&y<0,\\
\left(-1+\frac{r}{2}\right)y^2-rxy+\frac{r}{2}x^2,&y\geq0.\end{cases}$$
For any $r>2,$ both pieces of $\varphi_r$ are strictly convex quadratic. Thus, $e_rf_1(x)=\inf\phi_r(y)>-\infty$ and $r_1\leq2.$ For any $r<2,$ the second piece of $\varphi_r$ is concave quadratic, so $e_rf_1(x)=-\infty$ for all $x$ and $r_1\geq2.$ Therefore, $r_1=2.$ Similarly, $r_2=2.$ But $f(x)=x^2$ has $\bar{r}=0$ by Fact \ref{fact1}.\end{proof}
\begin{figure}[H]
\begin{center}
\includegraphics[width=0.4\textwidth]{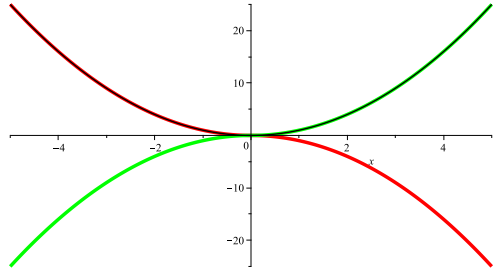}
\end{center}
\caption{$f_1$ (red) and $f_2$ (green) have threshold 2, but $f=\max\{f_1,f_2\}$ (black) has threshold 0.}\label{pb1fig}
\end{figure}
\noindent It is equally simple to construct an example where $\bar{r}=\max r_i$ for a piecewise function, for instance
$$f(x)=\begin{cases}f_1(x),&x\geq0,\\f_2(x),&x<0\end{cases}$$
where $f_1,f_2$ are defined in Example \ref{ex1}. In that case, $f(x)=-x^2$ and $\bar{r}=2=\max\{r_1,r_2\}.$ So we cannot do better than bounding $\bar{r}$ from above in this most general setting. Furthermore, one can obtain a prox-bounded function from the sum of two functions that are not prox-bounded. For instance, $f_1(x)=x^3$ and $f_2(x)=-x^3$ are not prox-bounded, but their sum is the constant function zero, with threshold zero. The next section considers more specific cases of both the sum and the composition of functions, where we can make some tighter conclusions about exact thresholds.

\section{Calculus of the threshold of prox-boundedness}\label{sec:calculus}

In this section, we consider the thresholds of the sum and the composition of prox-bounded functions. The functions here are no longer (necessarily) piecewise functions, as they were in the previous section. The following definition and facts will be useful.
\begin{df}[Fenchel conjugate]
For any function $f:\R^n\to\OR$, the \emph{Fenchel conjugate} of $f$ is the function $f^*:\R^n\to\OR$ defined by$$f^*(y)=\sup\limits_{x\in\R^n}\{\langle y,x\rangle-f(x)\}.$$
\end{df}
The Fenchel conjugate and the Moreau envelope enjoy a beautiful equivalence, as the following fact states.
\begin{fact}\emph{\cite[Proposition 2.1]{MR3783644}}\label{fenchfact}
For any proper function $f:\R^n\to\OR$,$$e_rf(x)=\frac{r}{2}\|x\|^2-g^*(rx),$$where $g(x)=f(x)+\frac{r}{2}\|x\|^2$.
\end{fact}
\begin{fact}\emph{\cite[Exercise 1.24]{rockwets}}\label{proxfact1}
For a proper, lsc function $f:\R^n\to\OR,$ the following are equivalent:
\begin{itemize}
\item[\rm(i)]$f$ is prox-bounded;
\item[\rm(ii)]$f$ majorizes a quadratic function;
\item[\rm(iii)]there exists $r\in\R$ such that $f+\frac{r}{2}\|\cdot\|^2$ is bounded below;
\item[\rm(iv)]$\liminf\limits_{\|x\|\to\infty}\frac{f(x)}{\|x\|^2}>-\infty.$
\end{itemize}If $\hat{r}$ is the infimum of all $r$ for which \emph{(iii)} holds, the limit in \emph{(iv)} is $-\frac{\hat{r}}{2}$ and the threshold for $f$ is $\bar{r}=\max\{0,\hat{r}\}.$
\end{fact}
First, we address the quadratic function mentioned in Fact \ref{proxfact1}(ii). Rockafellar and Wets state that it exists, but give no details as to its form. Lemma \ref{lem:proxfact1lem} below describes the curvature that such a quadratic function must have, in terms of the threshold.
\begin{lem}\label{lem:proxfact1lem}
Let $f:\R^n\to\OR$ be proper, lsc and prox-bounded with threshold $\bar{r}>0$. Then $f$ is bounded below by $-\frac{\bar{r}}{2}\|\cdot\|^2+m$ for some $m\in\R$. Furthermore, for any choice of $r<\bar{r}$ there does not exist $m\in\R$ such that $f$ is bounded below by $-\frac{r}{2}\|\cdot\|^2+m$. Therefore, $\frac{\bar{r}}{2}$ is the smallest possible curvature of a quadratic function that is a minorant of $f$.
\end{lem}
\begin{proof}
By Fact \ref{proxfact1}(iii), we have that $f+\frac{\bar{r}}{2}\|\cdot\|^2$ is bounded below, i.e., there exists $m\in\R$ such that for all $x\in\dom f$, $f(x)+\frac{\bar{r}}{2}\|x\|^2\geq m$. Rearranging, we have\begin{equation*}f(x)\geq-\frac{\bar{r}}{2}\|x\|^2+m\quad\forall x\in\dom f.\end{equation*}
Suppose that for some $r<\bar{r}$, there exists $m$ such that the above inequality holds, replacing $\bar{r}$ with $r$. Then we have that $f+\frac{r}{2}\|\cdot\|^2$ is bounded below, which by Fact \ref{proxfact1}(iii) and the postamble contradicts the fact that $\bar{r}$ is the threshold of $f$. Therefore, $\frac{\bar{r}}{2}$ is the minimum curvature of a quadratic minorant of $f$.
\end{proof}
\begin{rem}
In Lemma \ref{lem:proxfact1lem}, the condition $\bar{r}>0$ is necessary. We cannot make such a determination of the quadratic curvature in the case of $\bar{r}=0$, as there exist functions with threshold zero (such as affine functions) that are bounded below by a quadratic of curvature zero (i.e.\ affine function) and others that are not. For instance, the function $f:\R\to\R$, $f(x)=-|x|$ has threshold zero and is not bounded below by any affine function, but is bounded below by a concave quadratic function of any curvature greater than zero. We see this by noting that the inequality$$-|x|\geq-\frac{r}{2}x^2+m$$can be made true for all $x$ by making $m=m_r$ and shifting $m_r$ downwards as $r$ decreases.
\end{rem}
Now we focus on the threshold of the sum of two prox-bounded functions. As in the case of the piecewise function with prox-bounded pieces, we will find that an exact threshold cannot be obtained and we settle for an upper bound in the general case. If certain restrictions are put on one or both of the functions, an exact threshold can be determined.
\begin{lem}\label{lem:moreauinequality}
Let $f_1,f_2:\R^n\to\OR$ be proper, lsc and prox-bounded with respective thresholds $r_1,r_2$. If $f_1\leq f_2$, then for any $r>\max\{r_1,r_2\}$, we have$$e_rf_1\leq e_rf_2.$$
\end{lem}
\begin{proof}
Setting $g_1(x)=f_1(x)+\frac{r}{2}\|x\|^2$ and $g_2(x)=f_2(x)+\frac{r}{2}\|x\|^2$, we have $g_1\leq g_2$. By \cite[Proposition 13.14(ii)]{convmono}, $g_1^*(rx)\geq g_2^*(rx)$. By Fact \ref{fenchfact}, we have
\begin{align*}
e_rf_1(x)&=\frac{r}{2}\|x\|^2-g_1^*(rx),\\
&\leq \frac{r}{2}\|x\|^2-g_2^*(rx),\\
&=e_rf_2(x).\qedhere
\end{align*}
\end{proof}
\begin{cor}\label{rcor}
Let $f_1,f_2:\R^n\to\OR$ be proper, lsc and prox-bounded with respective thresholds $r_1,r_2$. If $f_1\leq f_2$, then $r_1\geq r_2$.
\end{cor}
\begin{prop}\label{boundedf2}
Let $f_1,f_2:\R^n\rightarrow\OR$ be proper, lsc and prox-bounded with respective thresholds $r_1,r_2$. Define $f(x)=(f_1+f_2)(x).$ Then $f$ is prox-bounded with threshold $\bar{r}\leq r_1+r_2.$ Moreover, if $f_2$ is bounded, then $\bar{r}=r_1.$
\end{prop}
\begin{proof}The first part of this proposition appears as part of \cite[Lemma 2.4]{proxave}, but we provide a full proof for the sake of completeness.
For any $\varepsilon>0,$ we have
\begin{align*}
e_{r_1+r_2+2\varepsilon}f(x)&=\inf\limits_y\left\{f_1(y)+f_2(y)+\frac{r_1+r_2+2\varepsilon}{2}\|y-x\|^2\right\}\\
&=\inf\limits_y\left\{\left[f_1(y)+\frac{r_1+\varepsilon}{2}\|y-x\|^2\right]+\left[f_2(y)+\frac{r_2+\varepsilon}{2}\|y-x\|^2\right]\right\}\\
&\geq\inf\limits_y\left\{f_1(y)+\frac{r_1+\varepsilon}{2}\|y-x\|^2\right\}+\inf\limits_y\left\{f_2(y)+\frac{r_2+\varepsilon}{2}\|y-x\|^2\right\}\\
&=e_{r_1+\varepsilon}f_1(x)+e_{r_2+\varepsilon}f_2(x)>-\infty~\forall x\in\R^n.
\end{align*}
This tells us that $f$ is prox-bounded and $\bar{r}\leq r_1+r_2$. Now suppose that $f_2$ is bounded. Since $f_2$ is bounded below, we have that $r_2=0$ by Fact \ref{fact1}. Hence, $\bar{r}\leq r_1+0=r_1.$ Since $f_2$ is bounded above, there exists $M\in\R$ such that $M\geq f_2(x)~\forall x.$ Suppose that $r_1>0.$ (Otherwise, trivially $\bar{r}=0=r_1.$) Then for any $r\in(0,r_1),$ we have
\begin{align*}
e_rf(x)&=\inf\limits_y\left\{f_1(y)+f_2(y)+\frac{r}{2}\|y-x\|^2\right\},\\
&\leq\inf\limits_y\left\{f_1(y)+M+\frac{r}{2}\|y-x\|^2\right\},\\
&=M+\inf\limits_y\left\{f_1(y)+\frac{r}{2}\|y-x\|^2\right\},\\
&=M+e_rf_1(x)=-\infty.
\end{align*}
Hence, $\bar{r}\geq r_1.$ Therefore, $\bar{r}=r_1.$
\end{proof}
\begin{cor}
Let $f_1,f_2:\R^n\to\OR$ be proper, lsc and prox-bounded with threshold 0. Then $(f_1+f_2)(x)$ is prox-bounded with threshold 0.
\end{cor}
The very strong condition of $f_2$ being bounded above and below in Proposition \ref{boundedf2} can be relaxed slightly, as the corollary below indicates, with the same proof as the proposition.
\begin{cor}
Let $f_1,f_2:\R^n\rightarrow\OR$ be proper, lsc and prox-bounded with respective thresholds $r_1,r_2$. Define $f(x)=(f_1+f_2)(x).$ If $r_2=0$ and $f_2$ is bounded above, then $\bar{r}=r_1.$
\end{cor}
\begin{prop}\label{affineprop}
Let $f_1:\R^n\to\OR$ be proper and lsc. Let $f_2:\R^n\to\OR$ be an affine function. Then $f_1+f_2$ is prox-bounded with threshold $r_1$ if and only if $f_1$ is prox-bounded with threshold $r_1$.
\end{prop}
\begin{proof}
We see in \cite[Lemma 3.6]{bavcak2010infimal} that the Moreau envelope of $f$ can be expressed as the sum of a quadratic function and a Moreau envelope of $f_1$ only, with $x$ plus a constant as the argument. Therefore, $e_rf$ exists if and only if $e_rf_1$ exists and we have that $f$ has the same threshold as $f_1$.
\end{proof}
Proposition \ref{affineprop} invites another slight relaxation of the condition on $f_2$ in Proposition \ref{boundedf2}.
\begin{cor}
Let $f_1,f_2:\R^n\to\OR$ be prox-bounded with respective thresholds $r_1,r_2$. Define $f=f_1+f_2$. If $r_2=0$ and $f_2$ is majorized by an affine function, then $f$ is prox-bounded with threshold $\bar{r}=r_1$.
\end{cor}
Now we move on to sufficient conditions for a composition rule. This is a difficult issue; one can construct examples of composition where the resulting threshold is any nonnegative number one desires, or even nonexistent. The following simple examples demonstrate.
\begin{ex}\label{compex1}
For $a,b\geq0,$ define $f_1,f_2:\R\rightarrow\R,$
$$f_1(x)=-bx,\qquad\qquad f_2(x)=-\frac{a}{2}x^2.$$
Then the threshold of $f_1\circ f_2$ is $\bar{r}_{12}=0,$ while the threshold of $f_2\circ f_1$ is $\bar{r}_{21}=ab^2.$
\end{ex}
\begin{proof}
Since $(f_1\circ f_2)(x)=\frac{ab}{2}x^2$ is bounded below, by Fact \ref{fact1} we have $\bar{r}_{12}=0.$ Since $(f_2\circ f_1)(x)=-\frac{ab^2}{2}x^2,$ by the same method as the proof of Example \ref{ex1} we find $\bar{r}_{21}=ab^2.$
\end{proof}
Example \ref{compex1} shows that with two basic prox-bounded functions it is possible to obtain a threshold for the composition that is any particular nonnegative number, by making appropriate choices of $a$ and $b$. The next example shows that we can just as easily use two prox-bounded functions to construct a function that is not prox-bounded.
\begin{ex}\label{compex2}
Define $f_1(x)=x^2,$ $f_2(x)=-x^2.$ Then $f_1\circ f_2$ has threshold $\bar{r}_{12}=0,$ while $f_2\circ f_1$ is not prox-bounded.
\end{ex}
\begin{proof}
Since $(f_1\circ f_2)(x)=x^4$ is bounded below, by Fact \ref{fact1} we have that $\bar{r}_{12}=0.$ Since $(f_2\circ f_1)(x)=-x^4,$ by Fact \ref{proxfact1}(ii), $f_2\circ f_1$ is not prox-bounded.
\end{proof}
Furthermore, one can compose two functions that are not prox-bounded to form a function that is prox-bounded. For instance, $f_1(x)=-x^3$ and $f_2(x)=\ln x$ are not prox-bounded, yet they yield the composition $(f_1\circ f_2)(x)=-\ln^3x,$ which is minorized by $-x^2$ and thus prox-bounded by Fact \ref{proxfact1}(ii). So what can we say about the thresholds of the composition of prox-bounded functions? As in the case of the sum of prox-bounded functions, if we restrict ourselves to certain classes of functions, we can make some conclusions. We start by listing a known fact that is used in the proof of the subsequent proposition.
\begin{fact}\emph{\cite[Lemma 2.4]{proxave}}\label{fact:lambda}
Let $f:\R^n\to\OR$ be proper, lsc and prox-bounded with threshold $\bar{r}$. Then for any $\lambda\geq0$, $\lambda f$ is prox-bounded with threshold $\lambda\bar{r}$.
\end{fact}
\begin{prop}
Let $f_1:\R^m\to\OR$ and $f_2:\R^n\to\R^m$ be prox-bounded with respective thresholds $r_1,r_2.$ Let $\ran f_2\subseteq\dom f_1$, define $f=f_1\circ f_2$ and denote the prox-threshold of $f$ as $\bar{r}$ when it exists. Then the following hold.
\begin{itemize}
\item[\rm(i)]If $f_1,f_2$ are Lipschitz continuous, then $\bar{r}=0$.
\item[\rm(ii)]If $f_1$ is affine: $f_1(x)=ax+b$ with $a\geq0$, then $\bar{r}=ar_2$.
\item[\rm(iii)]If $f_2$ is affine: $f_2(x)=ax+b$, then $\bar{r}=a^2r_1$.
\end{itemize}
\end{prop}
\begin{proof}
(i) Let $f_1$ be $K_1$-Lipschitz and $f_2$ be $K_2$-Lipschitz. Then
$$\|f_1(f_2(y))-f_1(f_2(x))\|\leq K_1\|f_2(y)-f_2(x)\|\leq K_1K_2\|y-x\|,$$which says that $f_1\circ f_2$ is $K_1K_2$-Lipschitz. By Proposition \ref{prop0}, $\bar{r}=0$.\medskip\\
(ii) We have $(f_1\circ f_2)(x)=af_2(x)+b$, which yields $\bar{r}=ar_2$ by Fact \ref{fact:lambda} and the fact that the vertical shift by $b$ has no impact on the threshold.\medskip\\
(iii) Let $r>r_1$. Then by Fact \ref{proxfact1}(iii), $f_1+\frac{r}{2}\|\cdot\|^2$ is bounded below. We have, for some $m\in\R$ and for all $x\in\dom f_1$,
\begin{align}
f_1(x)+\frac{r}{2}\|x\|^2&\geq m,\nonumber\\
f_1(ax+b)+\frac{r}{2}\|ax+b\|^2&\geq m,\nonumber\\
f_1(ax+b)+ar\langle x,b\rangle+\frac{a^2r}{2}\|x\|^2&\geq m-\frac{r}{2}\|b\|^2.\label{eq:affinecomp1}
\end{align}
Hence, $f_1(ax+b)+ar\langle x,b\rangle+\frac{a^2r}{2}\|x\|^2$ is bounded below. Since \eqref{eq:affinecomp1} is true for any arbitrary $r>r_1$, it is true for all $r>r_1$. By an identical argument, for any $r<r_1$ we have that $f_1(ax+b)+ar\langle x,b\rangle+\frac{a^2r}{2}\|x\|^2$ is not bounded below. Thus, $r_1$ is the infimum of all $r$ such that \eqref{eq:affinecomp1} is true. By Fact \ref{proxfact1}, the threshold of $f_1(a\cdot+b)+ar\langle\cdot,b\rangle$ is $a^2r_1$. By Proposition \ref{affineprop}, we conclude that the threshold of $f_1\circ f_2$ is $a^2r_1$.
\end{proof}
So we have that if both $f_1,f_2$ are Lipschitz continuous functions, or if one of $f_1,f_2$ is affine, then the threshold of the composition can be determined exactly. It is clear from Examples \ref{compex1} and \ref{compex2} that if one of $f_1,f_2$ is quadratic, chaos ensues. So far, it does not seem that other standard properties such as convexity and boundedness are any more promising in forming composition rules, not even in providing an upper bound for the threshold. We leave the further development of properties of the threshold of prox-boundedness to future consideration.

\section{Conclusion and future work}\label{sec:conc}

The threshold of prox-boundedness of the objective function of a minimization problem is an important value to take into consideration when implementing optimization algorithms. In this work, we have determined the threshold of Lipschitz functions and bounds on the threshold of piecewise functions. We established properties of thresholds of the sum and the composition of functions under certain conditions and shown that when we do not have these conditions, functions can be constructed so that the threshold of the sum or composition is any nonnegative number.\par
This paper is the first step in determining thresholds for larger classes of functions, with the long-term goal of improving the efficiency of optimization routines that are based in the proximal point algorithm. At the moment, the conditions imposed are quite heavy; the search continues for other well-behaved functions whose thresholds can be identified or at least bounded. The work done here regarding piecewise functions, together with the results of \cite{hare2014thresholds} on PLQ functions, should open the way for exploration of thresholds of composition classes such as the fully subamenable functions of \cite{mohammadi2019variational}. Such functions are an extension of fully amenable functions \cite{rockwets}, respect a chain rule and are likely suitable for use in constrained composite modelling and optimization applications.

\bibliographystyle{plain}
\bibliography{Bibliography}{}
\end{document}